\documentclass{amsart}
\usepackage{latexsym,amssymb,enumerate, amsmath}
\usepackage{color}

\newtheorem{theorem}{Theorem}[section]

\newtheorem{lemma}[theorem]{Lemma}
\newtheorem{proposition}[theorem]{Proposition}
\theoremstyle{definition}

\newtheorem*{ThmA}{Theorem A}

\newenvironment{enumeratei}{\begin{enumerate}[\upshape (a)]}
    {\end{enumerate}}

\def\irr#1{{\rm Irr}(#1)}
\def\cent#1#2{{\bf C}_{#1}(#2)}

\def\syl#1#2{{\rm Syl}_#1(#2)}

\def\nor{\trianglelefteq\,}
\def\oh#1#2{{\bf O}_{#1}(#2)}
\def\zent#1{{\bf Z}(#1)}

\def\aut#1{{\rm Aut}(#1)}
\def\out#1{{\rm Out}(#1)}

\def\fit#1{{\bf F}(#1)}
\def\frat#1{{\bf \Phi}(#1)}

\newcommand{\N}{{\mathbb N}}

\def\irr#1{{\rm Irr}(#1)}

\def\cent#1#2{{\bf C}_{#1}(#2)}

\def\syl#1#2{{\rm Syl}_#1(#2)}
\def\nor{\trianglelefteq\,}
\def\norm#1#2{{\bf N}_{#1}(#2)}
\def\oh#1#2{{\bf O}_{#1}(#2)}
\def\Oh#1#2{{\bf O}^{#1}(#2)}
\def\zent#1{{\bf Z}(#1)}

\def\aut#1{{\rm Aut}(#1)}

\def\out#1{{\rm Out}(#1)}

\def\fit#1{{\bf F}(#1)}
\def\lay#1{{\bf E}(#1)}
\def\fitg#1{{\bf F^*}(#1)}

\def\GF#1{{\rm GF}(#1)}
\def\SL#1{{\rm SL}_{2}(#1)}
\def\PSL#1{{\rm PSL}_{2}(#1)}
\def\GL#1#2{{\rm GL}_{#1}(#2)}

\def\V#1{{\rm V}(#1)}
\def\E#1{{\rm E}(#1)}

\def\irr#1{{\rm Irr}(#1)}

\def\cd#1{{\rm cd}(#1)}

\def\cent#1#2{{\bf C}_{#1}(#2)}
\def\syl#1#2{{\rm Syl}_#1(#2)}
\def\oh#1#2{{\bf O}_{#1}(#2)}
\def\Oh#1#2{{\bf O}^{#1}(#2)}
\def\zent#1{{\bf Z}(#1)}

\def\ker#1{{\rm ker}(#1)}
\def\norm#1#2{{\bf N}_{#1}(#2)}

\mathchardef\coso="2023
\def\ww#1{{#1}^{\coso}}
\def\cE{\overline{\rm E}}

\def \nq{\mathcal{N}_q}

\begin{document}

\title{On the character degree graph of finite groups}

\author[Z. Akhlaghi et al.]{Zeinab Akhlaghi}
\address{Zeinab Akhlaghi, Faculty of Math. and Computer Sci., \newline Amirkabir University of Technology (Tehran Polytechnic), 15914 Tehran, Iran.}
\email{z\_akhlaghi@aut.ac.ir}

\author[]{Carlo Casolo}
\address{Carlo Casolo, Dipartimento di Matematica e Informatica U. Dini,\newline
Universit\`a degli Studi di Firenze, viale Morgagni 67/a,
50134 Firenze, Italy.}
\email{carlo.casolo@unifi.it}

\author[]{Silvio Dolfi}
\address{Silvio Dolfi, Dipartimento di Matematica e Informatica U. Dini,\newline
Universit\`a degli Studi di Firenze, viale Morgagni 67/a,
50134 Firenze, Italy.}
\email{dolfi@math.unifi.it}

\author[]{Emanuele Pacifici}
\address{Emanuele Pacifici, Dipartimento di Matematica F. Enriques,
\newline Universit\`a degli Studi di Milano, via Saldini 50,
20133 Milano, Italy.}
\email{emanuele.pacifici@unimi.it}

\author[]{Lucia Sanus}
\address{Lucia Sanus, Departament d'\`Algebra, Facultat de
 Matem\`atiques, \newline
Universitat de Val\`encia,
46100 Burjassot, Val\`encia, Spain.}
\email{lucia.sanus@uv.es}

\thanks{The research of the second, third and fourth author is partially supported by the Italian PRIN 2015TW9LSR\_006 ``Group Theory and Applications' }


\begin{abstract} Given a finite group \(G\), let \(\cd G\) denote the set of degrees of the irreducible complex characters of \(G\). The \emph{character degree graph} of \(G\) is defined as the simple undirected graph whose vertices are the prime divisors of the numbers in \(\cd G\), two distinct vertices \(p\) and \(q\) being adjacent if and only if \(pq\) divides some number in \(\cd G\). In this paper, we consider the complement of the character degree graph, and we characterize the finite groups for which this complement graph is not bipartite. This extends the analysis of \cite{ACDKP}, where the solvable case was treated.
\end{abstract}

\maketitle

\section{Introduction}
Character Theory is a fundamental tool in the study of finite groups; in fact, it is well known that the set of the irreducible complex characters (the ordinary character table) of a finite group reflects the structure of the group very deeply. However, even a much smaller set of data, that can be extracted from the character table of a finite group \(G\), turns out to be very relevant: this is the \emph{degree set} \(\cd G\), whose elements are the degrees (i.e., the values at the identity element) of the irreducible characters of \(G\).

Starting from the famous Ito-Michler Theorem, which establishes that a given prime \(p\)  does not divide any number in \(\cd G\) if and only if \(G\) has an abelian normal Sylow \(p\)-subgroup, many results in the literature demonstrate the deep interplay between the group structure of \(G\) and the arithmetical structure of \(\cd G\). The \emph{character degree graph} (or \emph{degree graph}) \(\Delta(G)\) is a useful tool in order to capture the arithmetical structure of the degree set: this is the simple undirected graph whose vertices are the prime divisors of the numbers in \(\cd G\), and two (distinct) vertices \(p\), \(q\) are adjacent if and only if \(pq\) divides some number in \(\cd G\). 
The main questions in this research area concern the relationships between the group structure of \(G\) and certain graph-theoretical features of \(\Delta(G)\) (we refer the reader to the survey \cite{Lew}). 

One of the turning point for the investigation on the character degree graph is the ``Three-Vertex Theorem" by P.P. P\'alfy (\cite{PPP1}): if \(G\) is a finite \emph{solvable} group then, for any choice of three vertices of \(\Delta(G)\), two of these vertices are adjacent. Considering the complement \({\overline\Delta}(G)\) of \(\Delta(G)\) (i.e., the graph having the same vertex set, where two vertices are adjacent if and only if they are not adjacent in \(\Delta(G)\)), Palfy's theorem can be rephrased by saying that  \({\overline\Delta}(G)\) does not contain any triangle whenever \(G\) is a finite solvable group. In a recent paper (\cite{ACDKP}) this was extended by showing that, under the same solvability assumption, \({\overline\Delta}(G)\) does not contain any cycle of odd length, which is equivalent to say that \({\overline\Delta}(G)\) is a bipartite graph. 

The present work explores this context without the solvability assumption.
In general, the graph \({\overline\Delta}(G)\) may contain cycles of odd length; this happens, for instance, for any \(2\)-dimensional projective special linear group over a finite field in characteristic \(2\) with at least four elements (see Proposition~\ref{PSL2}). The point is that the existence of such a cycle restricts the structure of the group significantly. The main result of this paper characterizes this situation.

\begin{ThmA}
Let \(G\) be a finite group, and  let $\pi$ be a subset of the vertex set of \(\Delta(G)\) such that $|\pi|$ is an odd number larger than \(1\). Then $\pi$ is the set of vertices of a cycle in ${\overline{\Delta}}(G)$ if and only if
  $\Oh{\pi'}G = S \times A$, where  $A$ is abelian, $S \simeq \SL{u^{\alpha}}$ or $ S \simeq \PSL{u^{\alpha}}$ for a  prime $u \in \pi$ and a positive integer $\alpha$, and the primes in \(\pi\setminus\{u\}\) are alternately odd divisors of $u^{\alpha}+1$ and $u^{\alpha} -1$. 
\end{ThmA}

(In the statement above, \(\Oh{\pi'}G\) denotes the smallest normal subgroup of \(G\) whose index in \(G\) is a \(\pi'\)-number.) We remark that, in \cite{DKP}, the same class of groups is characterized by the stronger property that ${\overline{\Delta}}(G)$ contains a triangle, that is, \(\Delta(G)\) violates Palfy's ``three-vertex condition". Therefore, Theorem~A is an improvement of the main result of that paper, which in turn generalizes several previously known results concerning the character degree graph (see \cite[Corollaries B,C,D]{DKP}). Namely, the theorem by A. Moret\'o and P.H. Tiep (\cite{MT}) establishing that, for any choice of four vertices in \(\Delta(G)\), two of them are adjacent, is an immediate consequence of our results. The same holds for the theorem by M.L. Lewis and D.L. White that characterizes the finite groups whose degree graph has three connected components (\cite{LW1,LW2}), as well as  the bound on the diameter of the degree graph in the connected case (\cite{LW0}).

As another application of Theorem~A, we mention that it is possible to obtain a bound on the number of vertices of \(\Delta(G)\) in terms of the \emph{clique number} \(\omega(G)\) (i.e., the maximum size of a set of vertices inducing a complete subgraph) of \(\Delta(G)\). In particular it is shown in \cite{ACDPS} that, for any finite group \(G\), the size of the vertex set of \(\Delta(G)\) is bounded above by the largest number among \(2\omega(G)+1\) and \(3\omega(G)-4\), thus answering a question posed by the first author and H.P. Tong-Viet in \cite{ATV}.

A key tool for our proof of Theorem A is the analysis of certain linear actions of groups on finite modules, that is carried out in Section~3 (and introduced in Section~2). Another crucial step is the reduction obtained in Section~4 by means of a general lemma, that turns out to be very useful when handling statements about the degree graph by induction on the order of the group. This leads to the study of two special cases, treated in Section~5, and the proof of Theorem~A is then finished in Section~6. 

We conclude this introductory section by mentioning that the following discussion involves the classification of finite simple groups.

\section{Preliminaries}

Throughout the paper, every group is assumed to be a finite group. As customary, given a positive integer $n$, $\pi (n)$ denotes the set of all prime divisors of $n$, and, if $G$ is a group, we write $\pi (G)$ for  $\pi (|G|)$. For a given group \(G\), we denote by \(\Delta(G)\) the character degree graph as defined in the Introduction, and write $\overline{\Delta}(G)$ for the complement of \(\Delta(G)\); the set of vertices of $\Delta (G)$ is denoted by $\V G$, whereas we denote by $\E G$  and ${\overline{\rm{E}}}(G)$ the set of edges of $\Delta(G)$ and of $\overline{\Delta}(G)$, respectively. 

We start by recalling some properties of group actions on finite modules, that will be very relevant in our discussion.   
Let \(H\) and \(V\) be finite groups, and assume that \(H\) acts by automorphisms on \(V\). Given a prime number  \(q\), we say that the pair \((H,V)\) satisfies \(\nq\) if  $q$ divides $|H: \cent HV|$ and, for every non-trivial \(v\in V\), there exists a Sylow \(q\)-subgroup \(Q\) of \(H\) such that \(Q\trianglelefteq \cent H v\). 
We refer to~\cite{C} for a thorough analysis of this and related module actions.

If $(H, V)$ satisfies $\nq$ then, as recalled in the following lemma, \(V\) turns out to be an elementary abelian \(r\)-group for a suitable prime \(r\), and \(V\) is in fact an irreducible module for \(H\) over the field with \(r\) elements ${\rm GF}(r)$. 

\begin{lemma}
  \label{Nq1}
Let $(H,V) \in \nq$. Then $V$ is an elementary abelian $r$-group for a suitable prime $r$, and it  is an irreducible  $H$-module.
If $q \neq 2$, then $V$ is a primitive $H$-module.  
\end{lemma}

\begin{proof}
The first  assertion follows from Lemma~4 of~\cite{Z}. 
In the same Lemma, $V$ is also proved to be primitive under the additional 
assumption that $\Oh{q'}G = G$; we remove this assumption, still using essentially the same argument.

Assume $q \neq 2$. In order to prove that $V$ is primitive as an $H$-module, we can assume 
(by factoring out $\cent HV$) that $\cent HV = 1$. 
Working by contradiction, let $V = W_1 + W_2 + \cdots + W_m$ be an 
imprimitive decomposition of $V$, with $m \geq 2$, and let $K = \cap_{i = 1}^m \norm H{W_i}$
be the kernel of the transitive action of $H$ on the set $\{ W_1, W_2, \ldots, W_m \}$. 

Let $0 \neq w \in W_1$ and let $Q$ be  the unique  Sylow $q$-subgroup of $H$ such that $Q \leq \cent Hw$.   
For $1 \leq i \leq m$,  let $x_i \in H$ be such that $W_1^{x_i} =  W_i$. 
Let $v = w + w^{x_i}$ and  let $Q_i$ be the Sylow $q$-subgroup of $\cent Hv$. 
Thus $Q_i$ stabilizes the set $\{W_1, W_i\}$ and hence, as $q \neq 2$,  $Q_i$ 
stabilizes both $W_1$ and $W_i$. It follows that $Q_i \leq \cent Hw$ and hence $Q_i = Q$. 
Therefore $Q \leq K$ and  hence $(K, V) \in \nq$; but 
$V$ is not irreducible as $K$-module, since $m \geq 2$, a contradiction. 
\end{proof}

Let \(r\) be a prime number and \(n\) a positive integer; we denote  by $\Gamma(r^n)$ the semilinear group on  the field $\GF{r^n}$, and by $\Gamma_0(r^n)$ the subgroup of $\Gamma (r^n)$ induced by the field multiplications. 
Given a group $G$ and a faithful $G$-module $V$
over ${\rm GF}(r)$, $r$ a prime, we say that \emph{$G$ is a
semilinear group on $V$}, and write $G \leq \Gamma(V)$,  if there
exists an injective homomorphism $\phi: G \rightarrow \Gamma(r^n)$,
where $r^n = |V|$, such that the additive group $\GF{r^n}^+$ of
$\GF{r^n}$ (with the $G$-module structure  carried by $\phi$) and
\(V\) are isomorphic $G$-modules. 

\begin{theorem}
  \label{coprimeNq}
  Let $(G, V)$ satisfy $\nq$ and assume that either
\begin{enumeratei}
\item $(|V|, |G|) = 1$;  or
  \item $G$ has a normal $q$-complement; or
  \item $G$ is solvable and $q \neq 3$.
\end{enumeratei} 
Then $G/\cent GV$ is a 
subgroup of $\Gamma(V)$. 
\end{theorem}

\begin{proof}
(a) is  Theorem~2.5  and (b) is Theorem 2.2(a) of \cite{DKP}, while (c) is Corollary~10 of~\cite{C}.  
\end{proof}

Next, we gather some facts concerning the character degree graph of non-solvable groups. The next results are Propositions 2.6, 2.7, (part of)  2.8   and 2.10   of \cite{DKP}, respectively.  

\begin{proposition}\label{PSL2}
Let $S \simeq \PSL{u^{\alpha}}$ or $S \simeq \SL{u^{\alpha}}$, where $u$ is a prime and $\alpha \geq 1$. 
Let $\pi_{+} = \pi(u^{\alpha}+1)$ and $\pi_{-} = \pi(u^{\alpha}-1)$. For a subset $\pi$ of vertices 
of $\Delta(S)$, we denote by $\Delta_{\pi}$ the subgraph of $\Delta = \Delta(S)$ induced
by the subset $\pi$.
\begin{enumeratei}
\item If $u=2$, then $\Delta(S)$ has three connected components, $\{u\}$, $\Delta_{\pi_{+}}$ and 
$\Delta_{\pi_{-}}$, and each of them is a complete graph.  
\item If $u > 2$ and $u^{\alpha} > 5$, then  $\Delta(S)$ has two connected components, 
$\{u\}$ and  $\Delta_{\pi_{+} \cup \pi_{-}}$; also, both  $\Delta_{\pi_{+}}$ and $\Delta_{\pi_{-}}$ are
complete graphs, no vertex in $\pi_{+}\setminus\{2\}$ is adjacent to any vertex in  
$\pi_{-}\setminus\{2\}$, and $2$ is adjacent to all other vertices in $\Delta_{\pi_{+} \cup \pi_{-}}$. 
\end{enumeratei}
\end{proposition}

\begin{proposition}\label{W}
Let $G$ be an almost-simple group with socle $S \simeq \PSL{u^{\alpha}}$, where \(u\) is a prime. If \(s\) is a prime divisor of $|G/S|$, then \(s\) is adjacent in \(\Delta(G)\) to every prime in \(\pi(u^{2\alpha}-1)\).
\end{proposition}

\begin{proposition}\label{2.8}
  Let $G$ be an almost-simple group with socle $S$, and let $q$ and $p$ be distinct non-adjacent primes of $\Delta(G)$.
  If $q$ does not divide $|S|$, then $S$ is a simple group of Lie type in characteristic $p$.
  \end{proposition}

\begin{proposition}
  \label{CS}
  Let $G$ be a finite group, $M$ a non-abelian minimal normal subgroup of $G$ and $C = \cent GM$.
  Then the following conclusions hold.
  \begin{enumeratei}
  \item If $q$ is a prime divisor of $|G/MC|$ and $q$ does not divide $|M|$, then there exists $\theta \in \irr M$ such that $q$ divides $|G:I_G(\theta)|$.
  \item If $q$ is a prime divisor of $|G/C|$, then there exists $\theta \in \irr M$ such that
    $q$ divides $\chi(1)$ for all $\chi \in \irr{G|\theta}$.
  \item If $M$ is not a simple group, then $\Delta(G/C)$ is a complete graph. 
  \end{enumeratei}
\end{proposition}

Finally, we will freely use without references some notation and basic facts of Character Theory  such as  Clifford Correspondence, Gallagher's Theorem, Ito-Michler's Theorem, properties of character extensions and coprime actions (see \cite{Is}).

\section{About condition \(\nq\)}

This section contains a series of results concerning condition \(\nq\), that will be crucial in our proof of Lemma~\ref{specialcase}.

We remark that, if $(H, V) \in \nq$ and \(N\) is a normal subgroup of \(H\) such that \(q\) divides \(|N/\cent N V|\), then \((N,V)\) satisfies \(\nq\) as well.

\begin{lemma}
  \label{gl62}
Let $V$ be a faithful $H$-module, $|V| = 2^6$ and assume that 
$(H, V) \in \nq$. Then $q = 2$ and either $H \leq \Gamma(V)$ or $H$ has a normal subgroup isomorphic to $\SL 8$. 
\end{lemma}

We observe that $\rm{GL}_6(2)$ has a unique conjugacy class of maximal 
subgroups (isomorphic to ${\rm GL}_2(8) : C_3 = (C_7 \times {\rm L}_2(8)):C_3$) 
having a section isomorphic to ${\rm L}_2(8)$ and that if $H$ is a subgroup of $\rm{GL}_6(2)$ having a
normal subgroup isomorphic to $\SL 8$, then $(H, V) \in \nq$, where $V$ is the natural module of $\rm{GL}_6(2)$. 

\begin{proof}[Proof of Lemma~\(\ref{gl62}\)]
As $\cent HV = 1$, we can identify  $H$ with  
a  subgroup of $ G := {\rm GL}_6(2)$ and $V$ with the natural module.
We recall that $|G| = 2^{15} \cdot 3^4 \cdot 5 \cdot 7^2 \cdot 31$.

We first show that $q = 2$. Working by  contradiction, let $H$ be a subgroup of ${\rm L}_6(2)$ of minimal order  such that $(H, V) \in \nq$ for some  $q \neq 2$.
By the remark at the beginning of the page, $H$ has no proper normal subgroup of order divisible by~$q$. 

Let $Q$ be 
a Sylow $q$-subgroup of $H$. 
Considering that $n_q = |H: \norm HQ| = (2^6 -1)/(2^b -1)$ where $2^b = |\cent VQ|$, we see  (as $q\neq 2$)  that the 
only two possibilities are  $(q, n_q) = (5, 21)$ or $(q, n_q) = (31, 63)$. 
In the second case, $N = \norm HQ$ acts faithfully on $[V, Q]$ (as $V = \cent VQ \times [V, Q]$ and 
$N$ acts trivially on $\cent VQ$, as $|\cent VQ| = 2$). But $[V, Q]$ is an irreducible $Q$-module
(as  $31$ is a primitive divisor
of $2^5 -1$) and hence $N \leq \Gamma(2^5)$ by~\cite[Theorem 2.1]{MW}. 
Therefore, $|H| = 63\cdot |N|$ is odd and hence $H$ 
is solvable. Thus (by \cite[Corollary 10]{C}) $H \leq \Gamma(2^6)$, yielding that $q$ divides $6$, a contradiction. 

Assume now $q = 5$ and $[H: \norm HQ] = 21$. So, $Q$ has order $5$ and it is also a
Sylow $5$-subgroup of ${\rm GL}_6(2)$.  Observe that $\norm{{\rm GL}_6(2)}{Q} \simeq \Gamma(2^4) \times S_3$. 
Hence $|H|$ divides $2^3 \cdot 3^3 \cdot 5 \cdot 7$ and  $|H|$ is certainly a multiple of $3 \cdot 5 \cdot 7$. 
Let $K$ be a maximal normal subgroup of $H$; recall that $|K|$ is coprime to $5$.
If $H/K$ is solvable, then $H$ is $5$-nilpotent and by (b) of Theorem~\ref{coprimeNq} again $H \leq  \Gamma(2^6)$, 
a contradiction. 
So $H/K = S$ is a non-abelian simple group and $5$ divides $|S|$. Moreover, $7$ divides $|S|$, as otherwise
by the Frattini Argument $Q$ normalizes some Sylow $7$-subgroup $P$  of $H$ and hence 
$Q$ centralizes $P$ (because $|P|$ divides $7^2$), again a contradiction as 
$7$ divides $|H:\norm HQ|$.  

The non-abelian simple sections (i.e. factor groups  of subgroups)  of 
${\rm L}_6(2)$ are isomorphic to one of the groups in  the following list:

$$ \mathcal{L} =  \{{\rm A}_5, {\rm A}_6, {\rm A}_7, {\rm A}_8
, {\rm L}_2(7) 
, {\rm L}_2(8), {\rm U}_3(3) , {\rm U}_4(2), {\rm L}_3(4), {\rm L}_5(2), {\rm S}_6(2), {\rm L}_6(2) \} \; .$$

The only  group in $\mathcal{L}$ whose  order is at most  $2^3 \cdot 3^3 \cdot 5 \cdot 7$
and  is   divisible by both $5$ and $7$ is ${\rm A}_7$. But the normalizer in ${\rm A}_7$ of a Sylow 
$5$-subgroup has index $2 \cdot 3^2 \cdot 7$, and this is not possible as the same index  
divides $|H: \norm HQ| = 21$. 

\medskip

Assume now $q= 2$ and, as before, $H \leq G =  {\rm L}_6(2)$. 
We want to show either that  $H \leq \Gamma(2^6)$ or 
that $H$ has a normal subgroup isomorphic to ${\rm L}_2(8)$.

As above, (using (c) of Theorem~\ref{coprimeNq}) we can assume  that $H$ is non-solvable.
We can also assume that no proper normal subgroup of $H$ has
even order. In fact, by inspection of the maximal subgroups of ${\rm L}_6(2)$ one 
checks that if $N$ is a subgroup of ${\rm L}_6(2)$ and 
$N$ has a normal subgroup $L \simeq {\rm L}_2(8)$, then $L$ is characteristic in $N$ (as $L$ is the component
subgroup of $N$).

Write $K = \oh{2'}H$; by the preceding paragraph we see that
 $H/K = S$ is non-abelian simple and that $H$ is perfect. 
Let $Q$ be a Sylow $2$-subgroup of $H$. 
Observe that $n_2(S) = |\syl 2S|$ divides  $n_2(H) = |H: \norm HQ| = (2^6 -1)/(2^b -1) \in \{9, 21, 63 \}$. But for  $S$ in the above list of non-abelian simple sections of ${\rm L}_6(2)$, we check that the 
only groups $S$ that satisfy the  condition 
 $|\syl 2S|$  divides one among $\{9, 21, 63 \}$
are ${\rm L}_2(7)$ and ${\rm L}_2(8)$. 

We now proceed to show that $K$ is trivial. 
Let $P$ be a Sylow $p$-subgroup of $K$, for a prime $p \neq 2$. Then by 
Frattini Argument, $H = K \norm HP$ and hence $\norm HP/\norm KP \simeq H/K = S$. 
If the group ${\rm{Aut}}_H(P) = \norm HP/\cent HP$ is solvable, then 
$\cent HP \not\leq \norm KP$ and, being $\norm HP/\norm KP$ simple, it follows
that $\cent HP \norm KP = \norm HP$ and hence that $\cent HP$ has a factor 
group isomorphic to $S$ (in fact, $\cent HP/\cent KP \simeq \norm HP/\norm KP$).
So $\cent HP$ is a normal subgroup of even order of $H$ and hence $P$ is central in $H$. 
  Since the Schur Multiplier of $S$ has order $1$ or $2$ and $H$ is perfect, we conclude that $P =1$.

  Now, ${\rm{Aut}}_H(P)$ is certainly solvable if $ p \in \{5, 7, 31\}$ (observe that $7$ divides $|S|$).
  Moreover, a Sylow $3$-subgroup $Q$ of $G$ is isomorphic to the wreath product $C_3 \wr C_3$ and for
  all $P \leq Q$ one checks that ${\rm Aut}(P)$ is a $\{2,3\}$-group, except when $P$ is elementary abelian
  of order $3^3$; but as $13$ does not divide $|G|$, again ${\rm Aut}_H(P)$ is a $\{2,3\}$-group.
  Thus, we conclude that $K = 1$. 

Finally, if $H \simeq {\rm L}_2(7)$ and $Q \in \syl 2H$, then $\norm HQ = Q$ and
as $(H, V) \in \mathcal{N}_2$, we have that every subgroup of odd order of $H$ 
acts fixed point freely on $V$. But ${\rm L}_2(7)$ has  non-cyclic subgroups of order 
$21$, a contradiction (subgroups of order $pq$ of a Frobenius complement are cyclic). 
Thus, we conclude that $H \simeq {\rm L}_2(8)$.
\end{proof}

\begin{proposition}
\label{Nq2}
If $(H, V) \in \nq$, $\cent HV = 1$ and $q$ does not divide $|V|$, then the following conclusions hold.
\begin{enumeratei}
\item $\fit H$ is cyclic. 
\item If \(B\) is the solvable radical of \(H\) and \(H\neq B\) (i.e., if $H$ is non-solvable), then \(B=\fit H\). 
\end{enumeratei}
\end{proposition}

\begin{proof}
We start by proving (a), and we proceed by induction on $|H|$. 

Write $|V| = r^a$ and let $Q$ be a Sylow $q$-subgroup of $H$.
As $(H, V)  \in \nq$ and $r \neq q$, we have that 
$q \neq 2$, $a > 2$ and that $r^a \neq 2^6$ by Lemma~\ref{gl62}. 
So there exists a primitive prime divisor $t$ of $r^a - 1$ (note also that 
 $a$ divides $t-1$). 
Observe that $H$ has a subgroup $T$ of order $t$,  as 
$(|V|-1)/(|\cent VQ|-1) = |\syl qH|$ divides $|H|$. 

We can assume that $t$ does not divide the order of $F = \fit H$. 
In fact, assuming that $F$ has a non-trivial Sylow $t$-subgroup $T_0$, 
we have $H \leq \Gamma(V)$ (by~\cite[Theorem 2.1]{MW}), and $F$ is cyclic by~\cite[Lemma 3.7]{CDPS}. 

So, for every prime divisor $p$ of $|F|$, there is a $T$-stable Sylow $p$-subgroup $P$ of $F$. By Lemma 6 of \cite{C}, if \(p\neq 2\), then we have $P \leq \cent FT$; recalling that in \(\GL a r\) the centralizer of an element of order \(t\) is cyclic, we have that \(P\) is cyclic. Now, write $E = [\oh 2F, T]$ and observe that, by coprimality, $[E, T] = E$. 
If $E$ is cyclic, then $E=[E,T]=1$ (because the automorphism group of \(E\) is a \(2\)-group), so \(F\) centralizes \(T\) and it is therefore cyclic. 

Working by contradiction, assume that $E$ is non-cyclic. If \(A\) is a characteristic abelian subgroup of \(E\), then \(V\) is an irreducible \(AT\)-module (because \(t\) is a primitive prime divisor of \(r^a-1\)); moreover, the restriction of \(V\) to \(A\) is homogeneous, since otherwise it would have \(t\) homogeneous components, yielding the contradiction \(a\geq t>a\); thus, \(A\) is cyclic. As a consequence, $T$ centralizes every characteristic abelian subgroup of \(E\); since \([E,T]=E\), by \cite[24.7]{A} we have that $E$ is an extraspecial \(2\)-group. Write $|E| = 2^{2n +1}$, and let $z$ be the unique central involution of $E$. 

Observe that, by applying~\cite[Corollary 2.6]{MW} to an irreducible constituent of the homogeneous module $V_E$, we see that $2^n$ divides the dimension $a$ of $V$. 
From this we get $2^n +1 \leq a +1 \leq t$.

Moreover, $T$ acts faithfully on the symplectic module $E/\zent E$, so $t$ is a 
divisor of $2^{2i} - 1$ for some $i \leq n$. Therefore $t \leq 2^n +1$ and  
hence we conclude that $t = 2^n +1 = a+1$.  Thus, in particular, $V$ is an absolutely irreducible $E$-module (again by~\cite[Corollary 2.6]{MW}). 

Suppose first that $n \geq 2$. So there exists a non-central involution $x \in E$ and we set 
$W = \cent Vx$ and $X = \cent Hx$. 

We claim that $|W| = |V|^{1/2}$. In fact, if 
$\chi$ is the Brauer character corresponding to $V$, then $\chi$ is the 
only non-linear  character in $\irr E$ (note that $r \neq 2$, as $E$ acts 
faithfully and irreducibly on $V$), hence  $\chi(x) = 0$. 
So $[\chi_{\langle x \rangle}, 1_{\langle x \rangle}] = \chi(1)/2 = a/2$ is the 
dimension of the fixed-point space of $\langle x \rangle$ in $V$, and the 
claim is proved. 

So, by~\cite[Proposition~2.3(c)]{DKP}, $|W| > |\cent VQ|$ and  $q$ divides $|H:\cent H{W}|$. Thus, \cite[Lemma~2.4]{DKP} yields $(X,W) \in \nq$.
By induction, then $\fit{X/\cent XW}$ is cyclic. 
On the other hand, $X \cap E = \cent E{\langle x \rangle} = \langle x \rangle \times E_0$, 
where $E_0$ is an extraspecial group and $|E_0| = 2^{2(n-1)+1}$. 
Note that $\cent {E_0}W  = 1$, as $\cent {E_0}W$ is a normal subgroup of $E_0$ and it intersects trivially $ \zent{E_0} = \zent E =  \langle z \rangle$ (as $z$ acts as the 
inversion on $V$). So $E_0$ is isomorphic to a subgroup of $(X \cap E)\cent XW/\cent XW 
\leq \fit{X/\cent XW}$, which is cyclic, a contradiction. 

Hence, we have $n = 1$  and $a = 2$. 
Thus, we can identify $H$ with  a subgroup of ${\rm GL}_2(r)$.   
Let $K = H \cap {\rm SL}_2(r)$. If $(q, |K|) = 1$, then $H$ is $q$-nilpotent,
and then for every non-trivial $v \in V$, $\zent{\cent Hv}$ contains a Sylow $q$-subgroup of $H$. 
Therefore,  $H \leq \Gamma(V)$ by the Main Theorem of~\cite{C}. 
By~\cite[Lemma 2.1]{DKP} it follows $q =2$ and hence by~\cite[Proposition 2.3(a)]{DKP} we get
$r = 2 = q$, a contradiction. 

Assume now that $q$ divides $|K|$; then $(K, V) \in \nq$ and hence 
$K$ is a proper subgroup of ${\rm SL}_2(r)$ and $K \not\simeq A_5$ (as
$A_5 \simeq {\rm PSL}_2(5)$) by~\cite[Proposition 13]{C}. 
Hence $K$ and $H$ are  solvable and again $K$ and $H$ are subgroups of $\Gamma(V)$. 
Thus, as in the previous paragraph, we get $r = 2 = q$, a contradiction.    

\smallskip
It remains to prove part (b). Recalling that there exists a primitive prime divisor \(t\) for \(r^a-1\) that divides \(|H|\), we observe that \(t\) does not divide \(|B|\). In fact, assume that \(B\) has a non-trivial Sylow \(t\)-subgroup \(T_0\); then \(H=\norm H {T_0} B\). But \(\norm H {T_0}\) has an abelian normal subgroup (namely, \(\zent {T_0}\)) acting irreducibly on \(V\), hence \(\norm H {T_0}\) is metacyclic by \cite[Theorem~2.1]{MW}, contradicting the fact that \(H\) is not solvable.

Now, by \cite[Lemma 6]{C}, we have $B = \cent BT D$, where $D \in \syl 2 B$ is normalized by \(T\). 
  
Let $E = [B, T]$. Then $E$ is normal in $B$, and $E = [\cent BT D, T] = [D, T] \leq D$; therefore $E$ is a normal \(2\)-subgroup of \(B\). In particular, $E \leq\fit B=\fit H$ is cyclic by part (a). Finally, as the automorphism group of \(E\) is a \(2\)-group, \(T\) centralizes \(E\) and we get $[B,T]=[E,T] = 1$; thus $B \leq \cent HT$ is cyclic, and we get \(B=\fit H\), as claimed.  
\end{proof}

Next, we extend Proposition~13 of~\cite{C} to the case of even characteristic.  
\begin{lemma}
\label{Nq3}
Let $H=\PSL{u^{\alpha}}$, and \(M\) an \(H\)-module. Then, for any odd prime \(q\), there does not exist any \(H\)-module \(M\) such that $(H, M)$ satisfies $\nq$.  
\end{lemma}

\begin{proof}
  For a proof by contradiction, assume that \(M\) is an \(H\)-module and \(q\) is an odd prime such that \((H, M)\) satisfies \(\nq\). Then, by Proposition 13 of \cite{C}, we have \(u=2\).

Now, let $Q$ be a Sylow $q$-subgroup of $H$, and let $ r^a = |M|$, $r^b = |\cent MQ|$.  
Also, let $U$ be a Sylow $2$-subgroup of $H$ consisting of  the unipotent matrices
and let $D$ be the subgroup consisting of the diagonal matrices in $H$. 
Then $UD = \norm HU$ is a Frobenius group with complement $D$ and kernel $U$. 

We observe that $\cent MU = 1$: in fact, if $1 \neq v \in \cent MU$, then 
if $\cent Mv$ contains both $U$ and a Sylow $q$-subgroup $Q_1$ of $H$ as a 
normal subgroup; in particular, $U \times Q_1$ is a subgroup of $H$, which is impossible. 

As $\cent MU = 1$, then $r \neq 2$. Thus, by~\cite[15.16]{Is} we have that 
$\cent MD \neq 1$. Hence $D$ normalizes a Sylow $q$-subgroup  $Q$ of $H$ and, again 
by the knowledge of the subgroups of $H$, we get that $Q \leq D$ (and that $q$
divides $2^{\alpha}-1$). 
Hence, $|\norm HQ| = 2(2^{\alpha} -1)$ and 
$$|\syl qH|= 2^{{\alpha}-1}(2^{\alpha} +1) = \frac{r^a -1}{r^b -1} = 1 + r^b + \cdots + r^{b(c-1)}\; $$
Since ${\alpha} \geq 2$ (otherwise $H \simeq S_3$ and $q = 2$) and $r$ is odd, we get
that $c = a/b$ is even. 

Set $r^{a_0} = |\cent MD|$; using again~\cite[15.16]{Is}, we get 
$a = |D| a_0$ and $b = (|D|/|Q|)a_0$. Hence $c = |Q|$ and hence $q =2$, a 
contradiction.  
\end{proof}

\begin{lemma}
\label{Nq4}
Let $H=J_1$, the first Janko's sporadic simple group. Then, for any prime \(q\), there does not exist any \(H\)-module \(M\) such that $(H, M)$ satisfies $\nq$.
\end{lemma}
\begin{proof}
Let \(q\) be a prime divisor of \(|H|\). Then \(|\syl q H|\) can be deduced from \cite{atlas}, and it can be checked that the equation
$$|\syl qH| = \frac{r^a -1}{r^b -1} = 1 + r^b + \cdots + r^{b(c-1)}\; $$
is not satisfied by any choice of a prime \(r\) and positive integers \(a\), \(b\) with \(b\mid a\) (and \(c=a/b\)).
\end{proof}

\section{A reduction}

This section is devoted to the proof of a technical result, which provides a useful reduction for Proposition~\ref{SectionPSL}. Lemma~\ref{AmazingReduction} is a generalization of  \cite[Lemma 3.1]{ACDKP}:  we extend that result (which deals with solvable groups) to groups whose generalized Fitting subgroup coincides with the Fitting subgroup. 

\begin{lemma}
\label{AmazingReduction}
Let $G$ be a group such that $\fitg G=\fit G$, and such that for every proper factor group $\overline{G}$ of $G$, we have
$\V{\overline{G}} \neq \V G$.
Let $M$ be a minimal normal subgroup of $G$, let $p \in \V G \setminus \V{G/M}$ and $P\in\syl p G$. Also, let $\pi_p$ be the set of vertices of the connected component
of $p$ in $\overline{\Delta}(G)$.
Then there exists a normal subgroup $W$ of $G$ such that $\pi_p \subseteq \V{W}$,
and either $\fit W = P$ with $P' = M$, or $\fit W = M$ with $M\cap\frat G=1$;
in particular, $\fit W$ is complemented in $G$.
\end{lemma}

\begin{proof}
Let $\mathcal{N}_0$ be the set of the minimal normal subgroups of $G$ which are contained in $\frat G$.
For $N_0 \in \mathcal{N}_0$, if $p \in \V{G} \setminus \V{G/N_0}$ and $P \in \syl pG$, then
$PN_0/N_0$ is abelian and normal in $G/N_0$, so $P$ is normal in $G$ (as $\fit{G/N_0} = \fit{G}/N_0$) and
$P' = N_0$. Note that $p$ is therefore the unique prime in $\V{G} \setminus \V{G/N_0}$. In this situation, we define $\ww{N}_0= P$; obviously $\ww{N}_0$ is complemented in $G$ and, setting $F=\fit G$, there exists $K\nor G$ such that $F=K\times\ww{N}_0$.
On the other hand, let  $\mathcal{N}_1$ be the set of the minimal normal subgroups of $G$ that are \emph{not} contained in $\frat G$. If $N_1\in\mathcal{N}_1$, we define $\ww{N}_1 = N_1$: since  $\ww{N}_1 \cap \frat G = 1$, also in this case we have that $\ww{N}_1$ is complemented in $G$ (see 4.4 in \cite[III]{Hu}) and, as $\ww{N}_1 \leq \zent F$, there exists $K\nor G$ such that $F=K\times\ww{N}_1$.

Observe that $F$ is a direct product of subgroups $\ww N$, where $N$ varies in  $\mathcal{N}_0$ and in a suitable subset of $\mathcal{N}_1$. Moreover, if $\sigma$ is the set of prime divisors of $|N_0|$ for $N_0\in\mathcal{N}_0$, we have $\oh{{\sigma}'}{F}\cap\frat G=1$, and it is easily seen that $F$ itself has a complement $L$ in $G$.

Let now $M$ be a minimal normal subgroup of $G$. Take again $p \in \V{G} \setminus \V{G/M}$, $P \in \syl pG$, and let $\ww{M}$ be as above. As mentioned, we can write $F=K\times \ww M$ where $K$ is a suitable normal subgroup of $G$, thus $H = L\ww M$ is a complement for $K$ in $G$.

Define $W = \cent HK$. We have $W \nor G$,  $W \cap F = \ww M$ and hence $\fit W = \ww M$. Note also that, as $P$ commutes with $K$ modulo $M$, we get $[P, K] \leq M \cap K = 1$; in particular, $W$ contains the Sylow $p$-subgroups of $H$, and $p \in \V W$.

Let now $q\in \pi_p$ be a vertex of the connected component of $p$ in $\overline{\Delta}(G)$.  We prove, by induction on the distance $d = d_{\overline{\Delta}(G)}(p, q)$ that
 $q \in \V W$.

We consider first the case $d = 1$.
Given a character $\lambda \in \irr{\ww{M}}$ such that $\lambda_M \neq 1_M$, we have
that $p$ divides $\chi(1)$ for every $\chi \in \irr{G|\lambda}$.
Observe also that, both in the case $\ww M = P$ as also $\ww M = M$, $\lambda$ extends to
$I_G(\lambda)$.
Hence, Gallagher's Theorem and Clifford Correspondence, with
the nonadjacency between $q$ and $p$ in $\Delta(G)$, imply that $I_{KL}(\lambda)\simeq I_G(\lambda)/\ww M$ contains a
Sylow $q$-subgroup $Q_0$ of $LK$ as a normal subgroup, and that $Q_0$ is abelian.
Let now $Q$ be a Sylow $q$-subgroup of $L$; by a suitable choice of $\lambda$, we can assume that
$Q = Q_0 \cap L$.
Since $K$ centralizes $\ww M$,  $K$ is a normal subgroup of $I_{KL}(\lambda)$ as well.
Hence $Q$ centralizes
$K = (K \cap Q_0) \times \oh{q'}K$, i.e. $Q \leq W$ and so $q\in \V{W}$.
Observe that  $Q \neq 1$, as otherwise $(K \cap Q_0) \times M$ would contain an abelian normal Sylow
$q$-subgroup of $G$.

Assume now $d \geq 2$ and let $t$ be the vertex adjacent
to $q$ in a path of length $d$ from $p$ to $q$ in $\overline{\Delta}(G)$. 
Working by induction on $d$,  we have  $t \in \V{W}$. Let $T\in Syl_t(W)$.  Observe that, since $t\ne p$ and $\ww M=\fit W$,
  $T$ is certainly not normal in $W$.
Let $X = T^G=T^W$ (the normal closure of $T$ in $W$)  and $U = \Oh tX$.

  Let $U/V$ be a chief factor of $G$. Write $\overline G = G/V$ and use the bar convention.
  So $\overline U$ is a minimal normal subgroup of $\overline G$. Let $\overline C = \cent {\overline G}{\overline U}$.
  Note that if $\overline U$ is non-abelian, then we can assume that $|\overline U|$ is coprime to $q$, as otherwise certainly $q \in \V U \subseteq \V W$.
  On the other hand, if $\overline U$ is abelian, then $\cent {\overline U}{\overline T} = 1$ (since $\overline U = \cent {\overline U}{\overline T} \times [\overline U, \overline T]$ and
  $\overline T [\overline U, \overline T]$ is the normal closure of $\overline T$ in $\overline G$).  
  We observe  that there exists a character $\theta \in \irr {\overline U}$ such that $t$ divides $\chi(1)$ for all $\chi \in \irr{\overline G | \theta}$. 
  In fact, if $\overline U$ is abelian the prime $t$ divides the index $|\overline G: I_{\overline G}(\theta)|$
for all $1\ne \theta \in \irr {\overline U}$, while when $\overline U$ is non-abelian this follows from Proposition~\ref{CS}(b) (note that $t$ divides $|\overline G/\overline C|$ because
 ${\overline T}$ does not centralize \({\overline U}\); in fact, \([{\overline U},{\overline T}]=1\) would imply \({\overline X}={\overline T}\), whence \({\overline U}\) would be a \(t\)-group, against the fact that \(U= \Oh tX\)).
So, as $KU = K \times U$, we see that the Sylow $q$-subgroup of $K$ is abelian and that $K \simeq \overline K \leq \overline C$.

We also remark that $\overline U \cap \frat{\overline G} = 1$, as otherwise $\overline U\leq\frat{\overline G}$, so $\overline{T}$ is normal in $\overline G$ and 
$[\overline U,\overline T]=1$,  a contradiction.
If $\overline U$ is abelian, then $\overline U$ is complemented in $\overline G$ and $\theta$ extends to $I_{\overline G}(\theta)$.
So, by Gallagher's Theorem $I_{ \overline G}(\theta)/{ \overline U}$ contains a unique Sylow $q$-subgroup  of $\overline G$. 

Assume now that $\overline U$ is non-abelian. So $\overline U \overline C = \overline U \times \overline C$, and then we have that $\overline C$ has a normal abelian
Sylow $q$-subgroup.
If $q$ does not divide  $|\overline G/\overline C|$, then $\overline G$ has a normal abelian Sylow $q$-subgroup.
If  $q$ divides $|\overline G/\overline C|$, then
 by Proposition~\ref{CS}(c) and by Proposition~\ref{2.8} we deduce that $\overline U$ is a simple group of Lie type in characteristic $t$.
Let then  $\phi$ be the Steinberg character of $\overline U$. So $\phi$ is $\overline G$ invariant and it extends to $\overline G$ (since the corresponding character of
the almost simple group $\overline G / \overline C$ does). 
Hence, again  by Gallagher's Theorem $I_{ \overline G}(\theta)/{ \overline U}$ contains a unique Sylow $q$-subgroup  of $\overline G$. 

We have thus proved that, in all cases, there exists a Sylow $q$-subgroup $Q_0$ of $G$ such that $\overline{Q_0}\overline U /\overline U$ is the unique Sylow $q$-subgroup of
$I_{ \overline G}(\theta)/{ \overline U}$.

For every $\kappa \in \irr{\overline K}$ we have  $ \overline U \leq I_{\overline G}(\kappa)$; let us consider the character $\psi = \theta \times \kappa \in \irr{\overline U \times \overline K}$,
where $\theta \in \irr{\overline U}$ is as above. 
Note that $t$ divides $\chi(1)$ for every $\chi \in \irr{\overline G |\psi}$ (as $\chi$ lies over $\theta$).   As  $t$ and $q$ are not adjacent in $\Delta(\overline G)$,
we have that $I_{\overline G}(\psi) = I_{\overline G}(\theta) \cap I_{\overline G}(\kappa)$ contains a Sylow $q$-subgroup of $\overline G$.
As $\overline U \leq I_{\overline G}(\psi)$ and  $\overline{Q_0}\overline U /\overline U$ is the unique Sylow $q$-subgroup of
$I_{ \overline G}(\theta)/{ \overline U}$, 
this yields
$\overline{Q_0}\overline U \le I_{\overline {G}}(\psi) \le I_{\overline G}(\kappa)$, and this holds for every $\kappa \in \irr {\overline K}$.
Thus $\overline{Q_0}$ acts trivially on $\irr{\overline K}$. \
Now, by a suitable choice of (a conjugate of) $\theta$, we may assume that  $Q  = Q_0 \cap L \in \syl qL$.
Let $K_q$ and $K_{q'}$ be the Sylow $q$-subgroup and the $q$-complement of $K$, respectively, so that the nilpotent subgroup \(K\) decomposes as \(K_q\times K_{q'}\). Since $\overline Q$ acts trivially on
both $\irr{\overline{K_{q'}}}$ and $\irr{\overline{K_{q}}}$, we get  $[\overline Q, \overline{K_{q'}}] = 1$ by coprimality, but also $[\overline Q, \overline{K_{q}}] = 1$ by the fact that \({\overline K_q}\) is abelian. We conclude that $\overline{Q}$ acts trivially on $\overline K$. Thus $[Q, K ] \leq K \cap V = 1$ and hence $Q \leq W$. 

Finally, observe that $Q \neq 1$, as otherwise $F = K \times \ww M$ contains a full Sylow $q$-subgroup $Q_0$ of $G$ and $Q_0$  is abelian because both $Q_0 \cap K$ and
$Q_0 \cap \ww M$ are abelian (in fact, by the definition of $\ww M$, if $Q_0 \cap \ww M \neq 1$, then $\ww M$ is abelian).
As the Fitting subgroup $\ww M$ of $W$ has trivial intersection with $Q \leq L$, we conclude that $q \in \V W$, completing the proof.
\end{proof}

\section{Two special cases}

In view of Lemma~\ref{AmazingReduction}, two special situations turn out to be relevant in our analysis; these are treated in Lemma~\ref{specialcase} and Lemma~\ref{non-abelian}. For the first of these lemmas, the discussion about condition \(\nq\) that we carried out previously turns out to be critical.

\begin{lemma}
\label{specialcase}
Let \(G\) be a group, and let \(\pi=\{p_0,...,p_{d-1}\}\) be a subset of \(\V G\) such that \(d\neq 1\) is odd. Assume that \(\fit G=\fitg G\) is a minimal normal subgroup of \(G\), and that there exists a complement \(H\) for \(\fit G\) in \(G\). Assume also that \(\pi \not\subseteq \V H\). Then there is no cycle in $\overline{\Delta}(G)$ whose vertex set is   \(\pi\). 
\end{lemma}

\begin{proof}
Let \(p\) be a prime number such that \(\oh p H\neq 1\) (note that \(p\in\V G\)). Setting \(M=\fit G\), we have \(\cent M {\oh p H}\neq M\); but \(\cent M {\oh p H}\nor G\) and \(M\) is minimal normal in \(G\), thus we get \(\cent M {\oh p H}=1\). Denoting by \({\widehat M}\) the dual group of \(M\), of course we also have \(\cent{\widehat M} {\oh p H}=1\); as a consequence, for every \(\lambda\in{\widehat M}\setminus\{1_M\}\), the index of \(I_H(\lambda)\) in \(H\) is divisible by \(p\). Let now \(q\) be a vertex of \(\Delta(G)\) which is not adjacent to \(p\). As \(M\) is abelian and complemented in \(G\), every irreducible character of \(M\) extends to its inertia subgroup in \(G\) and so, by Clifford Theory, \(I_H(\lambda)\) contains a unique Sylow \(q\)-subgroup of \(H\) (which is also abelian and clearly non-trivial). In other words, \((H,{\widehat M})\) satisfies \(\nq\).

For a proof by contradiction, assume that \(\pi\) is the set of all  vertices of  a cycle in \({\overline{\Delta}}(G)\) (say, \(\{p_i, p_{i+1}\}\) is an edge of \({\overline{\Delta}}(G)\) for every \(i\in\{0,...,d-2\}\), as well as \(\{p_{d-1},p_0\}\)). By our hypotheses, we can assume that \(p_0\) is not a vertex of \(\Delta(H)\), whence \(H\) has an abelian normal Sylow \(p_0\)-subgroup. Since \(\oh{p_0} H\neq 1\), the paragraph above yields that \((H, {\widehat M})\) satisfies \(\nq\) for \(q\in\{p_1,p_{d-1}\}\); moreover, one among \(p_1\) and \(p_{d-1}\) does not divide \(|M|\) and, as \(M=\fitg G\), \(H\) acts faithfully on \({\widehat M}\). Therefore, denoting by \(B\) the solvable radical of \(H\), Proposition~\ref{Nq2} yields that \(B=\fit H\) is cyclic. 

Now, define \(\pi_0\) to be the set of primes in \(\pi\) that are divisors of \(|B|\) (note that \(p_0\in\pi_0\)), and let \(\pi_1=\pi\setminus\pi_0\). Also, setting \({\overline H}=H/B\), adopt the bar convention for \({\overline H}\). We claim that \emph{if \(\sigma\) is a subset of \(\pi_1\), with \(|\sigma|\neq 1\), yielding a path in \({\overline{\Delta}}(G)\), then there exists a minimal normal subgroup \({\overline N}\) of \({\overline H}\) which is (non-abelian and) simple, and such that \(\sigma\subseteq\pi(|{\overline H}:\cent{\overline H}{\overline N}|)\)}. In fact, let \(t\) be in \(\sigma\); since \(B\) is the solvable radical of \(H\), the generalized Fitting subgroup of \({\overline H}\) is a direct product of non-abelian minimal normal subgroups, and the intersection of all the centralizers of these subgroups is therefore trivial. As a consequence, there exists a non-abelian minimal normal subgroup \({\overline N}\) of \({\overline H}\) such that \(t\mid |{\overline H}:{\overline C}|\), where \({\overline C}=\cent{\overline H}{\overline N}\). Note that, by Proposition~\ref{CS}(b), there exists \(\theta\in\irr{\overline N}\) such that \(t\) divides \(\chi(1)\) for all \(\chi\in\irr{{\overline H}|\theta}\). Let now \(s\) ($s \neq t$) be an element of \(\sigma\) which is adjacent to \(t\) in \({\overline{\Delta}}(G)\), and assume \(s\mid |{\overline C}|\); since \(\oh s{\overline H}\) is clearly trivial, there exists \(\phi\in\irr{\overline C}\) whose degree is divisible by \(s\). Considering \(\theta\times\phi\in\irr{{\overline N}{\overline C}}\), we see that \(st\) divides \(\xi(1)\) for every \(\xi\in\irr{{\overline H}|\theta\times\phi}\), a contradiction. We conclude that also \(s\) is a divisor of \( |{\overline H}:{\overline C}|\) and, iterating this process, we get \(\sigma\subseteq\pi(|{\overline H}:{\overline C}|)\). The fact that \({\overline N}\) is simple follows by Proposition~\ref{CS}(c), and our claim is proved.

\medskip
As the next step, we prove that certain configurations concerning the distribution of the primes of \(\pi\) in the two parts \(\pi_0\) and \(\pi_1\) are not allowed. Namely:

\medskip
\noindent(a) \emph{Two consecutive vertices in \(\pi\) cannot be both in \(\pi_0\).} This follows from the fact that, \(B\) being the second Fitting subgroup of \(G\), the prime divisors of \(|B|\) induce a clique in \(\Delta(G)\) (see Proposition 17.3 in \cite{MW}). 

\medskip
\noindent(b) \emph{A path of consecutive vertices in \(\pi\), having the two (possibly coincident) extremes in \(\pi_0\) and all the other vertices in \(\pi_1\), cannot have length three}. In fact, arguing by contradiction, assume that \(t_1,t_2,t_3,t_4\) are consecutive vertices of a path in \({\overline{\Delta}}(G)\), with \(\{t_1,t_4\}\subseteq\pi_0\) and \(\{t_2,t_3\}\subseteq\pi_1\) (we allow \(t_1=t_4\)). Then, as shown in the first paragraph of this proof, \((H,{\widehat M})\) satisfies \(\nq\) for \(q\in\{t_2,t_3\}\). In other words, if \(\lambda\in{\widehat M}\setminus\{1_M\}\), then \(I_H(\lambda)\) contains both a Sylow \(t_2\)-subgroup and a Sylow \(t_3\)-subgroup of \(H\) as normal subgroups, and these Sylow subgroups are abelian; in particular, \(H\) has abelian Hall \(\{t_2,t_3\}\)-subgroups, and the same clearly holds for every normal section of \(H\). Take now a minimal normal subgroup  \({\overline N}\) of \({\overline H}\) which is simple and such that \(\{t_2,t_3\}\subseteq\pi(|{\overline H}:\cent{\overline H}{\overline N}|)\), as in the third paragraph of this proof; we have that \({\overline H}/\cent{\overline H}{\overline N}\) is an almost-simple group with abelian Hall \(\{t_2,t_3\}\)-subgroups,  such that \(t_2\) and \(t_3\) are nonadjacent vertices of \(\Delta({\overline H}/\cent{\overline H}{\overline N})\). This contradicts Proposition~2.8(b) of \cite{DKP}, and the desired conclusion follows.

\medskip
\noindent(c) \emph{A path of consecutive vertices in \(\pi\), having the two (possibly coincident) extremes in \(\pi_0\) and all the other vertices in \(\pi_1\), cannot have length larger than four}. For a proof by contradiction, assume that \(\{t_i\mid i\in\{1,...,k\}\}\) is a set of consecutive vertices of a path in \({\overline \Delta}(G)\), with \(\{t_1,t_k\}\subseteq\pi_0\), \(\{t_2,...,t_{k-1}\}\subseteq\pi_1\), and \(k> 5\) (we allow \(t_1=t_k\)). As in (b), we consider a minimal normal subgroup  \({\overline N}\) of \({\overline H}\) which is simple and such that \(\{t_2,...,t_{k-1}\}\subseteq\pi(|{\overline H}:\cent{\overline H}{\overline N}|)\). If a prime \(t\in\{t_3,...,t_{k-2}\}\) does not divide \(|{\overline N}|\) then, denoting by \(r\) and \(s\) the two vertices in \(\{t_2,...,t_{k-1}\}\) preceding and following \(t\) in that sequence, Proposition~2.8(b) of \cite{DKP} yields that \({\overline N}\) is a group of Lie type both in characteristic \(r\) and \(s\). It follows that one among \(r\) and \(s\) is \(2\), and \(t\) (which is now a prime divisor of \(|\out{\overline N}|\)) is \(2\) as well, a contradiction. For the same reason, we see that \(t_2\) and \(t_{k-1}\) cannot be both coprime with \(|{\overline N}|\). We conclude that either \(\{t_2,...,t_{k-1}\}\subseteq\pi({\overline N})\), or (say) \(t_2\) is the unique element of \(\{t_2,...,t_{k-1}\}\) not dividing the order of \({\overline N}\), which is in this case a group of Lie type in characteristic \(t_3\). 

An inspection of the degree graphs of non-abelian simple groups (which are all described in \cite{W}) yields that both situations imply \({\overline N}\simeq\PSL{u^{\alpha}}\) for a suitable prime \(u\) and \(\alpha\in\N\) (with \(u^{\alpha}\geq 4\)), or \({\overline N}\simeq J_1\), the first sporadic Janko's group. 

Let \(N\) be the normal subgroup of \(G\), containing \(B\), such that \(N/B={\overline N}\), and assume for the moment that \(B\) has a direct complement \(T\) in \(N\). By the first paragraph in this proof, we have that \((T,{\widehat M})\) satisfies \(\nq\) at least for \(q=t_{k-1}\), and also for \(q=t_{2}\) unless \(T\simeq{\overline N}\) is a group of Lie type in characteristic \(t_3\). In the former case, since one among \(t_2\) and \(t_{k-1}\) is an odd prime, the paragraph above together with Lemma~\ref{Nq3} and Lemma~\ref{Nq4} yield a contradiction. In the latter case, the same considerations show that we must have \(T\simeq\PSL{u^{\alpha}}\) for a suitable \(\alpha\in\N\), where \(u=t_3\) is an odd prime and \(t_{k-1}=2\); but even in this case we reach a contradiction, because the prime \(2\) is adjacent  in \(\Delta(\PSL{u^{\alpha}})\) (with \(u\neq 2\)) to every prime divisor of \(|\PSL{u^{\alpha}}|\) except \(u\), against the fact that \(t_{k-1}\) is adjacent to \(t_{k-2}\) in \({\overline \Delta}(G)\). 

It remains to consider the case in which \(B\) does not have a direct complement in \(N\). Note that \(\cent N B\) is a normal subgroup of \(N\) containing \(B\), therefore it is either \(B\) or \(N\); but if \(\cent N B=B\), then \({\overline N}\) embeds into the abelian group \(\aut B\) (recall that \(B\) is cyclic), a contradiction. Therefore we get \(B=\zent N\). Now, we have \(N=BN'\), thus \(N'\cap B\neq 1\) as otherwise \(N'\) would be a direct complement for \(B\) in \(N\); moreover, \(N'\) is perfect (see \cite[33.3]{A}), so \(N'\cap B\) is a non-trivial subgroup of the Schur multiplier of \(N'/(N'\cap B)\simeq {\overline N}\). As a consequence, we immediately exclude the case \({\overline N}\simeq J_1\) and \({\overline N}\simeq \PSL {2^{\alpha}}\) for \(\alpha\neq 2\), whose Shur multipliers are trivial. If \({\overline N}\simeq\PSL{u^{\alpha}}\) where \(u\) is an odd prime and \(u^{\alpha}\neq 9\), then we get \(p_0=2=|N'\cap B|\) (hence \(t_1=t_k=2\)); but \(2\) is adjacent in \(\Delta(G)\) to every prime divisor of \(|{\overline N}|\), clearly contradicting our setting. Finally, \(\PSL 4\simeq A_5\) and \(\PSL 9\simeq A_6\) have order divisible by only three primes, again contradicting our assumptions, and concluding the proof of (c).

\medskip
It is straightforward that conditions (a), (b) and (c) are not compatible with the fact that \(|\pi|\) is an odd number. This contradiction completes the proof.
\end{proof}

In order to treat (in Lemma~\ref{non-abelian}) the other relevant case arising from Lemma~\ref{AmazingReduction}, we gather together some facts that were essentially pointed out in \cite[Lemma~3.9]{CDPS} and in \cite[Lemma~2.2 and Lemma~3.2]{ACDKP}, but in a slightly different context (namely, those results deal with solvable groups). The proof of the following result is similar to those of the aforementioned lemmas (the main difference is that a key tool here is Theorem~\ref{coprimeNq}), but we decided to include it here in full details for the convenience of the reader. 

\begin{lemma}\label{new3.2} Let \(p\) be a prime, \(E\) an elementary abelian \(p\)-group, and \(H\) a \(p'\)-group acting faithfully on \(E\). Set \(G=EH\), and $\pi_0 = \pi (\fit H)$. Also, let $q \in\pi_0$ and $s\in \pi(H)$ be distinct vertices that are not adjacent in $\Delta(G)$. Let $Q=\oh q H, S\in \syl s H$ and $L = (QS)^H$ (the normal closure of QS in H). Then the following conclusions hold.	
\begin{enumeratei}
\item \(E=A\times B\), where \(A=[E,Q]=[E,L]\) and \(B=\cent EQ= \cent EL\). 
\item $L = L_0S \leq \Gamma(A)$, where $L_0 = \fit L$ is a cyclic group acting fixed-point freely and irreducibly on $A$.
\item Setting \(|A|=p^n\), there exists a primitive prime divisor \(t\) of \((p,n)\) such that \(t\) divides \(|L_0|\).
\item Let  $r$ be a prime divisor of $|H|$ such that  $\{s,r\} \in \cE(G)$. Then $r \in \pi_0$. Moreover, \(H\) has both a  normal abelian Sylow \(q\)-subgroup and Sylow \(r\)-subgroup. \item Let $p_1, p_2, p_3, p_4$ be distinct prime divisors of $|H|$ such that $\{ p_i, p_{i+1}\} \in \cE(G)$ for \(i\in\{1,2,3\}\), and $p_1\in \pi_0$. Then $\{p_1, p_4\} \in \cE(G)$.
\end{enumeratei}
\label{alternanza}
\end{lemma}

\begin{proof}
  Let $A = [E,Q]$ and $B = \cent EQ$.  
  Consider the action of $G$ on  $\widehat{E}=\widehat{A} \times\widehat{B}$, where $\widehat{E}=$Irr$(E)$ is the dual group of \(E\). We know that for  $\alpha \in   \widehat A \setminus \{1\}$, $q$ divides $|G : \cent G{\alpha}| = |H : \cent H{\alpha}|$. Also, the linear character $\alpha$ extends to $\cent G{\alpha}$, because $A$ has a complement (namely $B\cent H{\alpha})$ in $\cent G{\alpha}$. Thus, by Gallagher's Theorem and Clifford Correspondence, this forces $\cent H{\alpha} \simeq  \cent G{\alpha}/E$ to contain an $H$-conjugate of $S$ as a normal subgroup (and also, $S$ is abelian). Let $\alpha \in \widehat{A}\setminus\{1\}$ be such that $S \leq  \cent H{\alpha}$ and let $\beta \in \widehat{B}$; then $\cent H{\alpha\times \beta} =\cent H{\alpha} \cap \cent H{\beta}$  and $q$ divides $|H : \cent H{\alpha\times \beta}|$. As $\alpha\times \beta$ extends to its inertia subgroup in $G$, using as above Clifford Theory and that no irreducible character of $G$ has degree divisible by $qs$, we get that the unique Sylow $s$-subgroup $S$ of $\cent H{\alpha}$ must also be contained in $\cent H{\beta}$. We conclude that $S$ acts trivially on $\widehat{B}$ and hence that $S \leq \cent H B \unlhd H$. Thus $L = (QS)^H \leq \cent HB$, so that $B=\cent EL$. Moreover, we get $[E,L] =[A \times B,L] = [A,L]\leq A$, hence from $E = B \times [A, L]$ we deduce that 
$A = [E,L]$ and (a) is proved.

Since all Sylow $s$-subgroups of $H$ are contained in $L$, we have that $\cent L{\alpha}$
contains an $L$-conjugate of $S$ as a normal subgroup for every $\alpha \in \widehat{A}\setminus \{1\}$, i.e., $(L, \widehat{A})$ satisfies $\mathcal{N}_s$.  
Hence, as \(|L|\) is coprime to $|\widehat{A}|$, an application of Theorem~\ref{coprimeNq} yields that  $L \leq  \Gamma(A)$, as $\cent LA=1$.  Let $L_0=L \cap \Gamma_0(A)$.  By \cite[Lemma 2.1]{DKP}, we  have that $s$ does not divide $|L_0|$, whereas $d=|S|$ divides $n$, and $(p^n- 1)/(p^{n/d} -1)$ divides $|L_0|$.
Note that  $p^n-1$ has a  primitive prime divisor $t$. In fact,  otherwise, either $n =2$ or $p^n =2^6$. In both cases, as $s$ and $p$ are distinct primes, $s$ divides $(p^n-1)/(p^{n/d}-1)$, so $s$ divides $|L_0|$, a contradiction. This proves (c), and Lemma 3.7 in \cite{CDPS},  yields $L_0= \fit L$. 
		
Clearly $t$ divides $(p^n-1)/(p^{n/d}-1)$, hence it divides $|L_0|$. Denoting by $T_0$ the subgroup of $L_0$ with $|T_0|= t$, by Lemma 3.7 in \cite{CDPS},  it follows that $\cent L{T_0} = L_0$. Note that $t$ is larger than $n$ and hence, as $|L/L_0|$ divides $n$, we get that a Sylow $t$-subgroup of $L$ is contained in $L_0$. This implies that $Q \leq L_0$, since $Q \unlhd H$ centralizes $T_0$. But now both $Q$ and $S$ lie in $L_0S$; moreover, as $L/L_0$ is cyclic, $L_0S$ is normal in $H$. So $L= L_0S$.
As $L_0$ is cyclic and acts fixed point freely on $\widehat{A}$, the proof of (b) is complete. 

Set $K = \cent HA$, and observe that \(A\) is a faithful \(H/K\)-module over \(\GF p\). As clearly $L\cap K = 1$, claims (b) and (c) ensure that \(\fit{LK/K}\) has a characteristic subgroup \(T/K\) of order \(t\), which is therefore characteristic in \(LK/K\) and hence normal in \(H/K\). Since \(t\) is a primitive prime divisor of \((p,n)\), the action of \(T/K\) on \(A\) is easily seen to be irreducible, and therefore \(H/K\) embeds in \(\Gamma(A)\) (see Theorem~2.1 in \cite{MW}, for instance). As a consequence, denoting by \(X/K\) the Fitting subgroup of \(H/K\),
by~\cite[Lemma 2.1]{CDPS} the prime divisors of \(H/X\) constitute a clique in \(\Delta(H/K)\), which is a subgraph of \(\Delta(G)\); moreover, Lemma~3.7 in \cite{CDPS} yields that \(X/K\) is cyclic.
	
As $(L, \widehat{A})$ satisfies $\mathcal{N}_s$, it follows that $(H/K, \widehat{A})$ satisfies $\mathcal{N}_s$ and hence $\oh s{H/K} = 1$. In particular,  
\(s\) is a divisor of \(H/X\). Now, taking into account the conclusion of the paragraph above, if \(r\in\pi(H)\) is nonadjacent to \(s\) in \(\Delta(G)\), then \(r\nmid|H/X|\). 
As a consequence, \(H/K\) has a cyclic normal Sylow \(r\)-subgroup, thus \(r\not\in \V{H/K}\). Observe also that \(r\not\in \V K\) as well. In fact, certainly there exists \(\theta\in\irr L\) such that \(\theta(1)\) is divisible by \(s\); thus, if \(\phi\in\irr K\) has a degree divisible by \(r\), then, for any \(\xi\in\irr H\) such that \(\theta\times\phi\in\irr{LK}\) is a constituent of \(\xi_{LK}\), we would have \(r\cdot s\mid\xi(1)\), against the nonadjacency between \(r\) and \(s\) in \(\Delta(G)\). Now, let \(\chi\) be in \(\irr H\), and let \(\beta\) be an irreducible constituent of \(\chi_K\): we have \[\chi(1)=e\cdot\beta(1)\cdot|H:I_H(\beta)|,\] where \(e\) is the degree of an irreducible projective representation of \(I_H(\beta)/K\). If \(r\) divides \(|H:I_H(\beta)|\), then it also divides \(|H:I_H(\theta\times\beta)|\) (where \(\theta\in\irr L\) is as above), and therefore any \(\xi\in\irr{H\mid\theta\times\beta}\) would have a degree divisible by \(r\cdot s\), again a contradiction. Also, \( e\) is the degree of an irreducible ordinary representation of a Schur covering \(\Gamma\) of \(I_H(\beta)/K\); but, since \(I_H(\beta)/K\) has a cyclic normal Sylow \(r\)-subgroup, we have that \(\Gamma\) has an abelian normal Sylow \(r\)-subgroup, therefore \(r\) does not divide \(e\) as well. Since we observed that \(\beta(1)\) is not divisible by \(r\), we conclude that \(r\nmid\chi(1)\); as this holds for every \(\chi\in\irr H\), we get that \(H\) has an abelian normal Sylow \(r\)-subgroup, and the proof of  (d) is complete.
		
Let $p_1, p_2, p_3, p_4$ be as in the assumptions of part (e). Then (d) yields $\{p_1, p_3\}\subseteq \pi_0$
; actually, setting $P_i \in \syl {{p_i}}H$ for $i = 1, 2, 3, 4$, we have that \(P_1\) and \(P_3\) are abelian and normal in \(H\). Moreover, Lemma~2.1 of~\cite{CDPS} yields that $\{p_2, p_4\}\cap \pi_0=\emptyset$. 		
Write $L_1 = (P_1P_2)^H$, $L_3 = (P_3P_4)^H$, $L = L_1L_3$, $M=\cent EL$ and $V = E/M$. Also, let $U_1 = [ V, P_1]$ and $U_3 = [V, P_3]$. Then, clearly,  $V = U_1U_3$, and, again by (b) and (a), we have that $U_1 = [V, L_1]$ and $\ U_3 = [V,L_3]$ are irreducible $L$-modules, hence irreducible $H$-modules.
		
Let us assume $U_1\ne U_3$; then $V = U_1\times U_3$, $U_1 = \cent V{L_3}$ and $U_3 = \cent V{L_1}$. Thus we get $L_1\cap L_3 =1$, whence $L = L_1\times L_3$ and  $VL = U_1L_1\times U_3L_3$. As $p_2$ and $p_3$ divide, respectively, the degree of a character of $U_1L_1$ and of $U_3L_3$, we get that \(p_2\) and \(p_3\) are adjacent in the graph $\Delta (VL)$, which
is a subgraph of $\Delta(G)$, a contradiction. 		
Thus we have $U_1=U_3$. Setting \(U=[E,P_1]\), we now have $E = U\times M$.
So, $[E,P_1]= [E,P_3]$ and $p_1\cdot p_3$ divides $|H:I_H(\lambda)|$ for every non-principal $\lambda \in {\rm Irr}(U)$. Observe also that $U$ is complemented in $G$ by $MH$. 
Let \(\chi\) be an irreducible character of \(G\) whose degree is divisible by \(p_1\). Since $P_1$ is an abelian normal Sylow subgroup of $H$, then $MP_1$ is an abelian normal  subgroup of $MH\simeq G/U$, and hence the kernel of \(\chi\) cannot contain \(U\). Therefore, denoting by \(\lambda\) a (non-principal) irreducible constituent of \(\chi_U\), the degree of \(\chi\) is divisible by \(|G:I_G(\lambda)|=|H:I_H(\lambda)|\), thus by \(p_1\cdot p_3\). As \(p_4\) is not adjacent to \(p_3\) in \(\Delta(G)\), we conclude that \(p_4\) does not divide \(\chi(1)\). But this holds for every \(\chi\in\irr G\) such that \(p_1\mid\chi(1)\), hence $\{p_1, p_4\} \in \cE(G)$, as desired.
\end{proof}

We are now ready to prove the following lemma.

\begin{lemma}
\label{non-abelian}
Let \(G\) be a group, and let \(\pi=\{p_0,...,p_{d-1}\}\) be a subset of \(\V G\) such that \(d\neq 1\) is odd. Assume that \(\fitg G=P\), where \(P\) is a Sylow $p_0$-subgroup of \(G\) such that \(P'\) is a minimal normal subgroup of \(G\).
Then $\pi$ is not the vertex set of a cycle in $\overline{\Delta}(G)$. 
\end{lemma}

\begin{proof}
For a proof by contradiction, assume that the elements of \(\pi\) are the vertices of a cycle in \({\overline{\Delta}}(G)\) (i.e., $p_i$ is  adjacent to $p_{i+1}$ for every $i\in \{0,\ldots, d-1\}$ and also $p_0$ is adjacent to $p_{d-1}$); replacing \(\pi\) with a suitable subset, we can assume that no proper subset of \(\pi\), having odd size larger than \(1\), is the set of vertices of a cycle in \({\overline{\Delta}}(G)\) as well. Let  $H$ be a complement of $P=\fitg G$ in $G$. Write $p = p_0$.

We can assume  \(d\geq 5\), by looking at the main result in \cite{DKP}. We  claim that one among \(p_1\) and \(p_2\) lies in \(\pi_0=\pi(\fit H)\), where $H$ is a $p$-complement of \(G\). Assuming that \(p_1\) is not in \(\pi_0\), we will show  \(p_2\in\pi_0\).
	
Setting \(M=P'\), let \(\widehat{M}\) denote the dual group \(\irr M\), and let \(\lambda\) be a non-trivial element in \(\widehat{M}\). For any \(\tau\in\irr{P\mid\lambda}\) we get \(I_H(\tau)\leq I_H(\lambda)\), because \(M\leq\zent P\) and so \(\tau_M\) is a multiple of \(\lambda\); but we know that,  by coprimality, there exists \(\theta\in\irr{P\mid\lambda}\) for which in fact equality holds. Now, \[\cd{G\mid\theta}=\{|H:I_H(\lambda)|\cdot\theta(1)\cdot\xi(1)\;:\;\xi\in\irr{I_G(\lambda)/P}\},\] thus, the nonadjacency between \(p\) and \(p_1\) forces \(I_H(\lambda)\simeq I_G(\lambda)/P\) to contain a unique abelian Sylow \(p_1\)-subgroup of \(H\). 

Set \(K=\cent H M\) and note that, for \(\lambda\in\widehat{M}\setminus\{1_M\}\), \(K\) is a normal subgroup of \(I_H(\lambda)\). Denoting by \(P_1\) the unique Sylow \(p_1\)-subgroup of \(H\) contained in \(I_H(\lambda)\), we have that \(P_1\cap K\) is a normal Sylow \(p_1\)-subgroup of \(K\), whence it lies in \(\fit H\); but we are assuming \(p_1\not\in\pi_0\), thus \(p_1\nmid |K|\) and \(P_1\leq\cent H K\). As then $p_1$ divides $|H/K|$, an application of Theorem~\ref{coprimeNq} yields that \(H/K\) can be identified with a subgroup of the semilinear group \(\Gamma(\widehat{M})\).  Finally, setting \(X/K=\fit{H/K}\), the prime \(p_1\) divides the order of the cyclic group \(|H/X|\); in fact, if we assume the contrary, then \(P_1K=P_1\times K\) is normal in \(H\), and therefore \(P_1\leq\fit H\) against our assumptions. Observe that the prime divisors of \(|H/X|\) constitute a clique in \(\Delta(G)\); in particular, \(p_2\nmid|H/X|\).
	
Assume now that \(p_2\) is a divisor of \(|X/K|\). Hence, if \(P_2\) is a Sylow \(p_2\)-subgroup of \(H\), we get \(K<P_2K\unlhd H\) and \([M,P_2K]=M\). It follows that, for every \(\lambda\in\widehat M\setminus\{1_M\}\), \(p_2\) divides \(|H:I_H(\lambda)|\). In particular, \(p_2\) divides the degree of every irreducible character of \(G\) whose kernel does not contain \(M\). Now, there exists \(\chi\in\irr G\) such that \(\{p,p_3\}\subseteq\pi(\chi(1))\), and the kernel of such a \(\chi\) clearly does not contain \(M\). As a consequence, \(p_2\cdot p_3\) divides \(\chi(1)\), a contradiction. Our conclusion so far is that \(P_2\leq K\).
	
Consider now the normal closure \(P_1^H\) of \(P_1\) in \(H\), set \(Z=\zent K\) and \(L=P_1^HZ\unlhd H\). Recalling that \(P_1\leq\cent H K\unlhd H\), we get \(L\leq\cent H K\) as well, and so \(K\cap L=Z\).   In other words, \(KL/Z=(K/Z)\times(L/Z)\). Observe that \(p_1\) lies in \(\V{L/Z}\), as otherwise we would have \(P_1Z=P_1\times Z\unlhd L\), whence \(P_1\leq\fit H\) against our assumptions. Now, the nonadjacency between \(p_1\) and \(p_2\) in \(\Delta(G)\) forces \(p_2\not\in \V{K/Z}\). Therefore we get \(P_2Z/Z\unlhd H/Z\). But \(P_2Z\) is nilpotent, and we conclude that \(P_2\leq \fit H\), as we desired. So  either $p_1$ or $p_2$ lies in $\pi_0$. 
	
Now, adopting the bar convention for \(\overline G=G/\frat G\), we have that \(\overline P\) is an elementary abelian \(p\)-group, and \({\overline H}\simeq H\) is a \(p'\)-group acting faithfully (by conjugation) on \(\overline P\). 

If \(p_1\in\pi_0\), then Lemma~\ref{alternanza}(e) yields that $\{p_1, p_4\}  \in \cE(G)$, so we would have a shorter cycle in $\overline{\Delta}(G)$, a contradiction. On the other hand, if \(p_1\) does not lie in \(\pi_0\), then we have seen that \(p_2\) does, and so does \(p_{d-1}\) by an iterated application of Lemma~\ref{alternanza}(d). But again Lemma~\ref{alternanza}(e) yields now  $\{p_{d-4}, p_{d-1}\} \in \cE(G)$, which again gives a contradiction.
\end{proof}

\section{Proof of Theorem~A}

In this section we provide a proof for Theorem~A, the main result of this paper. A crucial step in this direction will be to show the following: if the complement of the degree graph of a group \(G\) contains a cycle of odd length larger than \(1\), then \(G\) has a composition factor isomorphic to $\PSL{u^{\alpha}}$, where \(u^{\alpha}\) is a suitable prime power. We start by proving this intermediate statement as a separate lemma.

\begin{proposition}
\label{SectionPSL}
Let $G$ be a group, and $\pi$ a subset of $\V G$ such that $|\pi|>1$ is an odd number. If  $\pi$ is the set of vertices of a cycle in the complement graph ${\overline{\Delta}}(G)$, then there exist two subnormal  subgroups $H$ and $K$ of $G$, with $K\trianglelefteq H$, such that $H/K$ is isomorphic to $\PSL{u^{\alpha}}$ for a suitable prime $u$ and positive integer $\alpha$ with $u^{\alpha}\geq 4$. Moreover, we have $\pi\subseteq\pi(H/K)$.
\end{proposition}

\begin{proof}
Let $G$ be a counterexample to the statement, having minimal order. We proceed through a number of steps.

\medskip
\noindent{\bf{(a)}} {\sl{$G$ has no non-abelian minimal normal subgroups.}}

For a proof by contradiction, let us assume that $M$ is a non-abelian minimal normal subgroup of $G$. Observe that, by our minimality assumption on $G$, there exists $p\in\pi$ which is not in $\V {G/M}$. Moreover, setting $C=\cent G M$, we have $p\in\pi(G/C)$, as otherwise $C$ would contain a Sylow $p$-subgroup $P$ of $G$; but $PM/M$ is an abelian normal subgroup of $G/M$, hence $P$ (which now intersects $M$ trivially) would be an abelian normal subgroup of $G$, against the fact that $p$ is in $\V G$.

Next we claim that, for every $r\in\pi(G/C)$, the prime $r$ is adjacent in $\Delta(G)$ to every vertex not lying in $\pi(G/C)$. In fact, let $t$ be in $\V G\setminus\pi(G/C)$ and let $T$ be a Sylow $t$-subgroup of $G$, which lies in $C$. Since $T$ cannot be abelian and normal in $C$, there exists an irreducible character $\psi$ of $C$ whose degree is divisible by $t$. Also, by Proposition~\ref{CS}(b), we can choose $\phi\in\irr M$ such that $r$ is a divisor of $\chi(1)$ for every $\chi\in\irr{G\mid \phi}$. If we consider an irreducible character $\chi$ of $G$ lying over $\phi\times \psi\in\irr{M\times C}$, we get $rt\mid\chi(1)$, as wanted.

Now, If $M$ is not a simple group, then Proposition~\ref{CS}(c) yields that $p$ is adjacent in $\Delta(G)$ to every element of $\pi(G/C)$ (note that, as $\fit{G/C}$ is trivial, we have $\V{G/C}=\pi(G/C)$). Therefore, as $p$ lies in $\pi(G/C)$ and in view of the previous paragraph, $p$ is in fact a complete vertex of $\Delta(G)$, clearly against our assumptions. As a consequence, $M$ must be simple, and ${\overline G}=G/C$ is an almost simple group with socle ${\overline M}$.

Recall that, by the hypothesis, there exist $q,s$ in $\pi$ that are not adjacent to $p$ in $\Delta(G)$ and, again by the claim proved in the second paragraph, both $q$ and $s$ are divisors of ${\overline G}$; it is then easy to verify that an iterated application of that claim yields $\pi\subseteq\pi({\overline G})$, and our minimality assumption forces now \(C=1\). Furthermore, assume that an element $r\in\pi$ does not divide $|M|$, and consider two elements $r^{-}, r^{+}$ in $\pi$ that are not adjacent to $r$ in $\Delta(G)$; then, by \cite[Proposition 2.8(b)]{DKP}, $M$ is a group of Lie type both in characteristic $r^{-}$ and $r^{+}$. On the other hand, for every such group $S$ (recall that here $S$ is isomorphic to one among $\PSL 4\simeq \PSL 5$, $\PSL 7\simeq {\rm{PSL}}_3(2)$, and ${\rm{PSU}}_4(2)\simeq{\rm{PSp}}_4(3)$), 
we have $|\out S|=2$, whence $r=2$ divides \(|M|\), a contradiction. The conclusion so far is that $\pi\subseteq{\pi(M)}$, and our minimality assumption on $G$ yields that $G=M$ is a simple group. But a recognition of the graphs $\Delta(S)$ where $S$ is a non-abelian simple group (see~\cite{W}) shows that the hypothesis of this theorem is satisfied if and only if $M$ is isomorphic to $\PSL{u^{\alpha}}$ for some prime $u$ and positive integer $\alpha$. In other words, $G$ is not a counterexample, and this is the final contradiction which completes the proof of Step (a).

\medskip
\noindent{\bf{(b)}} {\sl{We have $\fitg G=\fit G$.}}

Let $\lay G$ denote (as customary) the subgroup generated by the components of $G$ and, setting $Z=\zent{\lay G}$, assume $Z\neq 1$. By the minimality of $G$, there exists $p\in\pi$ such that $p$ is not a vertex of $\Delta(G/Z)$; as a consequence, if $P\in\syl p G$, then $PZ$ is a normal subgroup of $G$ and $PZ/Z$ is abelian. Since $PZ$ is solvable, it centralizes $\lay G$, and therefore $Z$ is central in $PZ$. We conclude that $P$ is normal in $G$, and now we must have $P\cap Z\neq 1$, as otherwise $P\simeq PZ/Z$ would be an abelian normal Sylow $p$-subgroup of $G$. Taking into account that $Z$ is the product of the centres of all the components of $G$, we deduce that for some component $K$ of $G$ we have $P\cap\zent K\neq 1$. On the other hand, as $P$ is normal in $G$ and $K/\zent K$ is a simple group, $p$ cannot be a divisor of $|K/\zent K|$; therefore we get $K=(P\cap\zent K)\times K_0$ (where $K_0$ is a $p$-complement of $K$), against the fact that $K=K'$. This contradiction shows that in fact $Z$=1, whence $\lay G$ is a product of non-abelian minimal normal subgroups of $G$. Step (a) implies of course $\lay G=1$, and the claim follows.

\medskip
Let $M$ be a minimal normal subgroup of $G$. We know that there exists $p\in\pi\setminus\V{G/M}$, and we consider a Sylow $p$-subgroup $P$ of $G$. Then the following holds.

\medskip
\noindent{\bf{(c)}} {\sl{Either $\fit G=M$ with $M\cap\frat G=1$, or $\fit G=P$ with $P'=M$.}}

This follows by the minimality of $G$, together with Lemma~\ref{AmazingReduction}. 

\medskip The final contradiction which concludes this proof is now achieved by means of Lemma~\ref{specialcase} and Lemma~\ref{non-abelian}, which exclude respectively the first and the second possibility of (c).
\end{proof}

Another important intermediate step in our proof of Theorem~A is provided by the following lemma.

\begin{lemma}
\label{newirreduciblemodule}
Let \(G\) be a group. Assume that \(K\) is an abelian minimal normal subgroup of \(G\) such that \(G/K\) is isomorphic either to \(\PSL{u^{\alpha}}\) or to \(\SL{u^{\alpha}}\) for a suitable prime $u$ and positive integer $\alpha$ with $u^{\alpha}\geq 4$. If \(K\) does not lie in \(\zent G\), then there exists \(\chi\in\irr G\) such that one of the following conclusions hold.
\begin{enumeratei}
\item \(\chi(1)\) is divisible by \(u\), and by either all the odd primes in \(\pi(u^{\alpha}+1)\) or all the odd primes in \(\pi(u^{\alpha}-1)\).
\item  \(\chi(1)\) is divisible by all the primes in \(\pi(u^{2\alpha}-1)\).
\end{enumeratei}
\end{lemma}

\begin{proof}
 
Set \(Z/K=\zent{G/K}\) (so, \(|Z/K|\leq 2\)). We start by observing that, in the case when \(Z/K\) is non-trivial, i.e., when \(G/K\) is isomorphic to \(\SL{u^{\alpha}}\) with \(u\neq 2\), every irreducible character of \(K\) has an extension to its inertia subgroup in \(G\). In fact, this is certainly true if \(K\) has a complement in \(G\). On the other hand, if \(K\) is not complemented in \(G\), then  $K \leq \frat G$, so \(Z\) is nilpotent; now, if \(K\) is not a \(2\)-group, then \(Z=K\times Z_0\) where \(Z_0\) is a normal subgroup of order \(2\) of \(G\), and we conclude the proof applying induction to the factor group \(G/Z_0\). Thus, to prove of our claim, we can assume that \(K\) is a $2$-group. Recall that $G/K\simeq\SL{u^a}$ contains just one involution and that a Sylow $2$-subgroup \(Q\) of \(G/K\) is a generalized quaternion group. An application of \cite[V.25.3]{Hu} yields that the Schur representation group of \(Q\) is \(Q\) itself; as a consequence, for each 
$\lambda \in \irr K$, the group $I_G(\lambda)/\ker{\lambda}$ splits over its central subgroup $K/\ker{\lambda}$, and again  $\lambda$ extends to  $I_G(\lambda)$, as claimed. 

Next, as we are assuming \(K\not\leq \zent G\), there exists \(\lambda\in\irr K\) such that \(I_G(\lambda)\) is a proper subgroup of \(G\); observe that, setting  \(I=I_G(\lambda)Z\), this \(I\) is a proper subgroup of \(G\) as well, because \(G/K\) is perfect. As a consequence, $I/Z$ is isomorphic to a proper subgroup of $\PSL{u^{\alpha}}$. Recall that the subgroups of $\PSL{u^{\alpha}}$ are of the following types (see~\cite[II.8.27]{Hu}), where $d = (2, u^{\alpha} -1)$:
    (i) dihedral groups of order $(2(u^{\alpha} \pm 1))/d$ and their subgroups;
    (ii) semidirect products of elementary abelian groups of order $u^{\alpha}$ by cyclic groups of order $(u^{\alpha} -1)/d$ and
    their subgroups;
    (iii) $A_4$ (unless $u=2$ and $\alpha$ is odd); $S_4$ (when $u^{2\alpha} \equiv 1 \pmod{16}$);
    $A_5$ (when $u^{\alpha}(u^{2{\alpha}} - 1) \equiv 0 \pmod 5$);
    (iv) $\PSL{u^{\beta}}$ or ${\rm PGL}(u^{\beta})$, where $\beta$ divides $\alpha$ (for $u \neq 2$, $\PSL{u^{\alpha}}$ has a subgroup
    isomorphic to ${\rm PGL}(u^{\beta})$ if and only if $\alpha/\beta$ is even).  

If \(I/Z\) is of type (i), then \(\pi(|G:I|)\) contains either \(\{u\}\cup\pi(u^{\alpha}+1)\) or \(\{u\}\cup\pi(u^{\alpha}-1)\), and conclusion (a) of our statement is satisfied by any \(\chi\in\irr{G\mid\lambda}\). 

As for type (iii), if \(I/Z\simeq A_4\) or \(I/Z\simeq S_4\), then the same argument as in the paragraph above yields conclusion (a), except possibly when \(u=2\) and \(u\) does not divide \(|G:I|\). 

In the case \(I/Z\simeq S_4\), this means \(G/K\simeq\PSL{2^3}\), but \(\PSL {2^3}\) does not have any subgroup isomorphic to \(S_4\). On the other hand, 
if \(I/Z\simeq A_4\), then we have to consider the case \(G/Z\simeq\PSL{2^{2}}\); observe that here, since \(u=2\), we have \(Z=K\) and \(I=I_G(\lambda)\), and we have \(\pi(u^{2\alpha}-1)=\{3,5\}\). Now, the Schur representation group of \(I/Z\) is a \(\{2,3\}\)-group. If \(\irr{I\mid\lambda}\) contains a linear character, then Gallagher's Theorem yields that there exists \(\psi\in\irr{I\mid\lambda}\) such that \(\psi(1)=3\); moreover, any non-linear character in \(\irr{I\mid\lambda}\) has a degree divisible by \(2\) or \(3\). In any case, 
there exists \(\chi\in\irr{G\mid\lambda}\) whose degree is divisible by either \(2\cdot 5\) or by \(3\cdot 5\), and conclusion (a) (resp. (b)) is satisfied.

In order to rule out type (iii), it remains to consider the case \(I/Z\simeq A_5\). In this situation, certainly \(u\) divides \(|G:I|\), and conclusion (a) holds unless possibly when the \(3\)-part of \(u^{\alpha}+1\) or \(u^{\alpha}-1\) is precisely \(3\). Therefore, as can be easily seen,  it is enough to show that \(\irr{I_G(\lambda)\mid\lambda}\) contains a character of degree divisible by~\(3\); to this end, we can assume that \(\lambda\) does not extend to \(I_G(\lambda)\), so (by the first paragraph of this proof) we are dealing with a situation in which \(Z=K\) and \(I=I_G(\lambda)\). Set \(K_0=\ker{\lambda}\), and observe that \(I/K_0\) does not split over \(K/K_0\), as otherwise \(\lambda\) would have an extension to \(I\); moreover, \(K/K_0\) is central in \(I/K_0\). As a consequence, we get \(I/K_0\simeq \SL{5}\). Now, the characters in \(\irr{I\mid\lambda}\) are precisely the faithful characters of \(I/K_0\); among those, there is a character of degree \(6\), as wanted.
      
Consider now type (iv): \(I/Z\simeq\PSL{u^{\beta}}\) or ${\rm PGL}_2(u^{\beta})$, where \(\beta\) divides \(\alpha\) (note that we already considered the case \(u^{\beta}\in\{4,5\}\), i.e., \(I/Z\simeq A_5\)). Clearly we have that \(u\) divides \(|G:I|\).
If $u^{\beta} = 9$, then $I/Z \simeq \PSL{9} \simeq A_6$ or $I/Z\simeq{\rm PGL}_2(9)$, and we have \(u=3\); but since only one odd prime (namely \(5\)) appears in \(\pi(3^2+1)\cup\pi(3^2-1)\), certainly conclusion (a) is satisfied in this case. 
So we can assume \(u^{\beta}\neq 9\). Now, \(I_G(\lambda)/K\) has a normal section isomorphic to \(\PSL{u^{\beta}}\), so it has irreducible characters whose degrees are divisible by \(u^{\beta}+1\) or by \(u^{\beta}-1\). If \(\lambda\) has an extension to \(I_G(\lambda)\), then conclusion (a) easily follows by Gallagher's Theorem. As a consequence, again we can assume \(Z=K\).
Let $I_0$ be the normal subgroup of $I=I_G(\lambda)$ (with $|I:I_0| \in \{1,2\}$) such that
$I_0/Z \simeq \PSL{u^{\beta}}$ and let $K_0 = \ker{\lambda}$. If $I_0/K_0$ does not split over $K/K_0$, then
$K/K_0 \leq (I_0/K_0)'$ and, as $K/K_0 \leq \zent{I_0/K_0}$ as well,
we deduce that $I_0/K_0 \simeq \SL{u^{\beta}}$, the representation group of
$I_0/K$. But $\SL{u^{\beta}}$ has faithful characters of degree $u^{\beta}+1$ (as well as
$u^{\beta} -1$), and hence there exists $\theta \in \irr{I_0|\lambda}$ such that
$u^{\beta}+1$ (or \(u^{\beta}-1\)) divides $\theta(1)$. If $I_0/K_0$ splits over $K/K_0$, then one gets the
same conclusion by Gallagher's Theorem, and also in this case conclusion (a) follows. 

Finally, assume that $I/Z$ is of type (ii), so \(u^{\alpha}+1\) divides \(|G:I|\). Let us assume that conclusion (b) is not satisfied; then there exists \(s\in\pi(u^{\alpha}-1)\) such that \(s\) does not divide \(|G:I|\). Now, $I_G(\lambda)/K$ has a normal section isomorphic to \(I/Z\), which is a Frobenius group whose kernel is an elementary abelian \(u\)-group, and whose complements have order necessarily divisible by \(s\) (note that we can assume the Frobenius kernel to have order divisible by \(u^{\alpha}\), as otherwise conclusion (a) holds). 
Since $I/Z$ has irreducible characters of degree divisible by $s$, we get that  $\lambda$  
does not extend to $I_G(\lambda)$, so \(Z=K\) and \(I_G(\lambda)=I\). Therefore, as every Sylow subgroup of $I/K$ other than the Sylow \(u\)-subgroup $U/K$ is cyclic, 
\cite[(11.31)]{Is} yields that $\lambda$ does not extend to $U$. 
As a consequence, if $\beta \in \irr{I\mid\lambda}$, then $u$ divides $\beta(1)$ and 
conclusion (a) follows.
\end{proof}

We are now in a position to conclude the proof of the theorem, which we state again.

\begin{ThmA}
Let \(G\) be a group, and  let $\pi \subseteq \V G$ be such that $|\pi|$ is an odd number larger than \(1\). Then $\pi$ is the set of vertices of a cycle in ${\overline{\Delta}}(G)$ if and only if
  $\Oh{\pi'}G = S \times A$, where  $A$ is abelian, $S \simeq \SL{u^{\alpha}}$ or $ S \simeq \PSL{u^{\alpha}}$ for a  prime $u \in \pi$ and a positive integer $\alpha$, and the primes in \(\pi\setminus\{u\}\) are alternately odd divisors of $u^{\alpha}+1$ and $u^{\alpha} -1$. 
\end{ThmA}

\begin{proof}
We first observe that, by~\cite[Corollary (11.29)]{Is}, the $\pi$-parts of the irreducible character degrees of $G$ and of $\Oh{\pi'}G$ are the same. Thus, we can henceforth assume that $G = \Oh{\pi'}G$; as the ``if part" of our claim easily follows by the structure of \(\Delta(\PSL{u^{\alpha}})\) for \(u^{\alpha}>3\) (see Proposition~\ref{PSL2}), we will focus on the ``only if part". 

 
By Proposition~\ref{SectionPSL}, we know that \(G\) has a composition factor \(H/K\) (which is indeed, clearly, a chief factor) such that \(H/K\simeq\PSL{u^{\alpha}}$ for a suitable prime power \(u^{\alpha}>3\), with \(\pi\subseteq\pi(H/K)\). Note that, in view of the structure of \(\Delta(H/K)\), our assumptions imply that \(u\) lies in \(\pi\) and, if \(q,s\) are the two ``neighbors" of \(u\) in \(\pi\), then (say) \(q\) divides \(u^{\alpha}+1\), \(s\) divides \(u^{\alpha}-1\), and \(q,s\) are both odd  (in fact, if \(u\neq 2\), then all the primes in \(\pi\) are necessarily odd); moreover, the primes in \(\pi\setminus\{u\}\) along the cycle are alternately odd divisors of $u^{\alpha}+1$ and $u^{\alpha} -1$, so the last claim is already proved.

As a consequence, \(G\) does not have any irreducible character \(\chi\) as in conclusion~ (a) or conclusion (b) of Lemma~\ref{newirreduciblemodule}. 

Our next claim is that every chief series of \(G\) has precisely one non-abelian chief factor whose order is divisible by some prime in \(\pi\); as a consequence, we will get that both \(K\) and \(G/H\) are \(\pi\)-solvable groups.

 Arguing by contradiction assume that, in a chief series containing $H$ and $K$, there is  another non-abelian chief factor $U/V$ of $G$ with $\pi \cap \pi(U/V) \neq \emptyset$, and let \(t\) be a prime in $\pi \cap \pi(U/V)$.
If $H \leq V$, then $\cent{G/K}{H/K}$ has a normal section isomorphic to $U/V$ (since ${\rm Out}(H/K)$ is solvable) and hence we get an edge between \(t\) and any primes in $\pi\setminus\{t\}$ in the subgraph $\Delta(H/K \times \cent{G/K}{H/K})$ of $\Delta(G)$, a contradiction.
If $U \leq K$, then  similarly (using also Proposition~\ref{CS}(c)) we get that $\cent{G/V}{U/V}$ has a normal section isomorphic to $H/K$ and we get a  contradiction in a similar way. 

Observe that, replacing \(H\) with a suitable term of its derived series, we can assume $H=H'$.

Let $C = \cent G{H/K}$; so  $G/C$ is an almost-simple group with socle isomorphic to $H/K$ and, by Proposition~\ref{W}, every  prime divisor of $|G/CH|$ is adjacent in $\Delta(G)$ to all the primes in both $\pi(u^{\alpha} -1)$ and $\pi(u^{\alpha} +1)$. Therefore, $G/CH$ is a $\pi'$-group. As $\Oh{\pi'}G = G$, we get $CH = G$. Thus we have $G/K = H/K \times C/K$, and this forces $\pi \cap \V{C/K} = \emptyset$; as a consequence, $C/K$ has a normal abelian Hall $\pi$-subgroup $D/K$, and again the assumption $\Oh{\pi'}G = G$ yields $C = D$. 
    
    In order to complete the proof,  we can assume that $G = H$ is perfect. 
In fact, \(H\) certainly satisfies our hypothesis, so, if \(H<G\), 
induction yields that \(H\) is isomorphic either to \(\PSL{u^{\alpha}}\) or to \(\SL{u^{\alpha}}\). 
Now, if \(2\in\pi\) (which implies \(2=u\)), then \(K=1\), and we get \(G=H\times D\); on the other hand, if \(2\not\in\pi\) and \(D_0\) is the Hall \(2'\)-subgroup of \(D\), it can be easily checked that \(G=H\times D_0\). In any case, the desired conclusion is satisfied by \(G\) and we are done.

Note that we can also assume $K > 1$, as $G/K \simeq \PSL{u^{\alpha}}$. We will show that then $G \simeq \SL{u^{\alpha}}$, 
with $u \neq 2$. Let $M < K$  be such that $K/M$ is a chief factor of $G$. 

Assume first $M = 1$, thus \(K\) is a minimal normal subgroup of \(G\). As \(G\) is perfect, it will be enough to show that $K \subseteq \zent G$: in this case, $G$ will be the Schur representation group of $G/K$, whence $G \simeq \SL{u^a}$. We observe that, in this situation, \(K\) is abelian. In fact, if \(K\) is non-abelian, then we know that \(K\) is a \(\pi'\)-group (see the third paragraph of this proof); therefore, as $\Oh{\pi'}G =G$, we cannot have \(G=K\cent G K\), and this forces \(\cent G K=1\). Now, \(K\) must be a simple group by Proposition~\ref{CS}(c) and so \(G/K\leq\out K\) would be solvable, which is not the case.  Thus \(K\) is (elementary) abelian, as claimed, and we are in a position to apply Lemma~\ref{newirreduciblemodule} in order to get the desired conclusion.

Assume finally that $M > 1$: then induction yields that 
$G/M \simeq \SL{u^a}$ with $u \neq 2$, and that $M$ is minimal normal in $G$. 
Again, essentially by the same argument as in the paragraph above, $M$ is abelian. Observe that \(M\) is certainly not central in \(G\), as the Schur multiplier of \(G/M\) is trivial, and now another application of Lemma~\ref{newirreduciblemodule} (with \(M\) in the role of \(K\)) yields the contradiction that completes the proof. 
\end{proof}

\end{document}